\documentclass[a4paper,12pt]{amsart}
\usepackage{float}
\usepackage{amsfonts,amsthm,amssymb,amsmath}
\usepackage{graphicx}
\usepackage{color}
\usepackage[all,cmtip]{xy}

\usepackage{t1enc}
\usepackage[latin1]{inputenc}

\newtheorem{theorem}{Theorem}[section]
\newtheorem{proposition}[theorem]{Proposition}
\newtheorem{corollary}[theorem]{Corollary}
\newtheorem{lemma}[theorem]{Lemma}

\newtheorem{remark}[theorem]{Remark}

\newcommand{\C}{\mathbb{C}}

\newcommand{\N}{\mathbb{N}}

\newcommand{\U}{\mathfrak A}
\newcommand{\B}{\mathcal{L}}
\newcommand{\E}{\mathcal{E}}

\newcommand{\I}{\mathcal{I}}
\newcommand{\M}{\mathcal{M}}
\newcommand{\s}{\mathcal{S}}

\newcommand{\fin}{\overline{\mathfrak{F}}}

\newcommand{\ep}{\ell_{p}}
\newcommand{\eq}{\ell_{q}}
\newcommand{\en}{\boldsymbol{\ell_{n}}}
\newcommand{\ei}{\ell_{\infty}}

\newcommand{\iso}{\overset 1 =}
\newcommand{\inc}{\overset {\leq1} \hookrightarrow}

\DeclareMathOperator{\spanned}{span}

\begin{document}

\title{Diagonal multilinear operators on K\"{o}the sequence spaces}

\author{Ver\'onica Dimant
\and
Rom\'an Villafa\~ne}

\thanks{Both authors were partially supported by CONICET PIP 2014-0483 and ANPCyT PICT 2015-2299. The second author was also partially supported by UBACyT 20020130100474 BA}

\date{}


\address{Departamento de Matem\'{a}tica y Ciencias, Universidad de San
Andr\'{e}s, Vito Dumas 284, (B1644BID) Victoria, Buenos Aires,
Argentina and CONICET} \email{vero@udesa.edu.ar}
\address{Departamento de Matem\'{a}tica - Pab I,
Facultad de Cs. Exactas y Naturales, Universidad de Buenos Aires,
(C1428EGA) Buenos Aires, Argentina and IMAS -  CONICET.}
\email{rvillafa@dm.uba.ar}

\begin{abstract}
We analyze the interplay between maximal/minimal/adjoint ideals of multilinear operators (between  sequence spaces) and their associated K\"{o}the sequence spaces. We establish relationships with  spaces of multipliers and apply these results to describe diagonal multilinear operators from Lorentz sequence spaces. We also define and study some properties of the ideal of $(E,p)$-summing multilinear mappings, a natural extension of the linear ideal of absolutely $(E,p)$-summing operators.
\end{abstract}

\maketitle

\section*{Introduction}

Trying to describe the connections between different ideals of linear operators (and the internal structure of them),  it began, in the 70's, the study of diagonal linear operators on $\ell_p$ spaces with the work of Carl \cite{Carl76}, K\"{o}nig \cite{Kon75} and Pietsch \cite{Pie80}. By means of limit orders they have compared  different ideals of linear operators and described their diagonal elements. Next, this research  continued in the context of K\"{o}the sequence spaces, leading to the so called \textit{multipliers}. This notion has its root in harmonic analysis where it has appeared within the study of Fourier series and Fourier transformation. Later, it has also been employed in many other contexts as, closer to our framework, Banach function spaces and Banach sequence spaces \cite{DMM06,BG87,AF01,Ka07,KLM13,KLM14}.

The concept of ideal of multilinear operators was also introduced by
Pietsch in \cite{Pie84} and it has been developed by several authors since then. Even if the multilinear theory has its roots in the linear one, it had  its own development  that led to
different situations involving new interesting techniques. Some usual linear ideals (absolutely $p$-summing operators, for instance) have many natural diverse extensions to the multilinear setting which enrich the theory by showing interesting situations that do not appear in the linear context (see, for instance \cite{PG05,CP07,PS11}). We refer to \cite{FG03,FH02} for general results
about ideals of multilinear mappings. For a presentation of the multilinear theory focused in the interplay with polynomials and holomorphic mappings the books \cite{Din99} and \cite{Muj86} are the classical references.

The introduction of limit orders for studying ideals of multilinear forms appeared in \cite{CDS06} and similar methods were used for general sequence spaces in \cite{CDS09}. There, it was defined the K\"{o}the sequence space associated to an ideal of multilinear forms acting on K\"{o}the sequence spaces. Later, in \cite{CDSV14}, this kind of study reached vector valued multilinear ideals between $\ell_p$-spaces.

Here, we propose a more general approach of the relationship between ideals of multilinear operators (acting on K\"{o}the sequence spaces) and their respective associated sequence spaces.  In particular, we analyze if for a maximal (minimal) ideal its associated sequence space is also maximal (minimal). In addition, we relate the sequence spaces associated to an ideal and its adjoint.

The spaces of multipliers appear to give us new descriptions of our sequence spaces associated to multilinear ideals. As an application, we can characterize diagonal multilinear operators from Lorentz sequence spaces.

In the final section, we define  the ideal of $(E,p)$-summing multilinear mappings, as a natural extension of the linear ideal of absolutely $(E,p)$-summing operators. We obtain some properties of this multilinear ideal by means of our previous results on associated sequence spaces.

\section{Preliminaries}\label{preliminares}

Throughout the paper we will use standard notation of the Banach
space theory. We will consider complex Banach spaces $E,F, \ldots$
and their duals will be denoted by $E', F',\ldots$. We will write $E=F$ if they are topologically isomorphic and $E\overset{1}{=}F$ if they are isometrically isomorphic. The symbol
$\overset 1 \hookrightarrow $ means an isometric injection and the symbol
$\inc$ means a norm one inclusion (not necessarily isometric).

Sequences of complex numbers will be denoted by $x =
(x(k))_{k=1}^{\infty}$, where each $x(k) \in \mathbb{C}$.
 By a \emph{K\"{o}the sequence space} (also known as Banach sequence space) we mean a Banach space $E \subseteq
\mathbb{C}^{\mathbb{N}}$ such that
$\ell_{1} \inc E \inc \ell_{\infty}$ and with the \emph{normal property}: if
$x \in \mathbb{C}^{\mathbb{N}}$ and $y \in E$ satisfy $\vert
x(k) \vert \leq \vert y(k) \vert$ for all $k \in \mathbb{N}$ then
$x \in E$ and $\Vert x \Vert \leq \Vert y \Vert$. Note that in a K\"{o}the sequence space $E$, given $x\in E$ and a sequence of complex numbers $s$ with $|s(k)|=1$ for all $k\in\N$, we should have $s\cdot x\in E$ and $\|s\cdot x\|=\|x\|$ (where the product is coordinatewise). For each $N \in \mathbb{N}$, we consider the $N$-dimensional
truncation $E_{N}:= \spanned \{e_{1}, \dots, e_{N}\}$
(where $e_n$ denotes the $n$-th canonical unit vector: $e_n(k)=\delta_{n,k}$ for all $k$).
 The canonical inclusion and projection will be denoted by
$i_{N} : E_{N} \overset 1 \hookrightarrow E$ and  $\pi_{N} : E
\twoheadrightarrow E_{N}$.

The K\"othe dual of a K\"{o}the sequence space $E$, defined as
\[
E^{\times} : = \left\{ z \in \mathbb{C}^{\mathbb{N}}  \colon \sum_{j \in
\mathbb{N}} |z(j)\cdot x(j)| < \infty \mbox{ for all } x \in E \right\},
\]
is a K\"{o}the sequence space with the norm
\[\|z\|_{E^{\times}} :=
\sup_{\| x \|_{E} \leq 1} \sum_{j \in \mathbb{N}} |z(j) \cdot x(j)|=\sup_{\| x \|_{E} \leq 1}\|z\cdot x\|_{\ell_1}.\]
It is well known (see, for example,~\cite[Lemma 2.8]{BeSha88}) that $z \in
E^{\times}$ if and only if the series $\sum_{j \in \mathbb{N}} z(j)\cdot x(j)$ converges for all $x \in E$. Also, $\|z\|_{E^{\times}} = \sup_{\|x\|_{E}\leq 1} \left|\sum_{j \in \mathbb{N}} z(j) \cdot x(j)\right|.$
Note that
$(E_{N})' \overset{1}{=} (E^{\times})_{N}$ holds for every $N$. In the same way that we define the K\"othe dual, we can considerate $(E^\times)^\times=E^{\times\times}$ and we say that $E$ is \emph{Köthe reflexive} if $E^{\times \times}\overset 1 = E$.

Following \cite[1.d]{LTII79}, a K\"{o}the sequence space $E$
is said to be \emph{$r$-convex} (with $1 \leq r < \infty$) if there
exists a constant $\kappa > 0$ such that for any choice $x_{1},
\dots, x_{m} \in E$ we have
\[
\bigg\| \bigg( \Big( \sum_{j=1}^{m} | x_{j}(k)|^{r} \Big)^{1/r} \bigg)_{k=1}^{\infty} \bigg\|_{E}
\leq \kappa \ \bigg( \sum_{j=1}^{m} \| x_{j}\|_{E}^{r} \bigg)^{1/r}.
\]
We denote by $\mathbf{M}^{(r)} (E)$ the smallest
constant which satisfies the inequality.

The \emph{minimal kernel} of a K\"{o}the sequence space $E$ is defined as the set
 \[ E^{min}:=\big\{x\in\ei \ : \ x=y\cdot z \ \  \textrm{ with } \ y\in E \ \ \textrm{and} \ \  z\in c_0\big\}\]
  which is also a K\"{o}the sequence space if we endow it with the norm
  \[ \|x\|_{E^{min}}=\inf\big\{\|y\|_E\cdot\|z\|_{\ei} \  : \ x=y\cdot z \ \  \textrm{ with } \ y\in E \ \ \textrm{and} \ \  z\in c_0 \big\}.\]
The \emph{maximal hull} of a K\"{o}the sequence space $E$ is defined as the set
  \[E^{max}:=\big\{x\in\ei \  :\  x\cdot z\in E \ \textrm{ for all }\  z\in c_0 \big\},\]
 which results a K\"{o}the sequence space if the norm is given by
  \[\|x\|_{E^{max}}=\sup_{z\in B_{c_0}} \|x\cdot z\|_E.\]
A K\"{o}the sequence space $E$ is said to be \emph{maximal} if $E\overset 1 = E^{max}$ and \emph{minimal} if $E\overset 1 = E^{min}$. For example, a K\"{o}the dual $E^\times$ is always maximal.

For a detailed study and general facts about K\"{o}the sequence spaces,
see \cite{LTI77, LTII79}.

\medskip

The space of continuous linear operators between two Banach
spaces $E$ and $F$ will be denoted by $\mathcal{L} (E;F)$ and the space of
continuous $n$-linear mappings from $E_{1} \times \cdots \times E_{n}$ to
$F$ by $\mathcal{L} (E_{1}, \ldots, E_{n};F)$. This is a Banach space  with the usual sup
norm, given by
$
\Vert T \Vert := \sup \left\{ \Vert T(x_{1} ,  \ldots , x_{n}) \Vert_{F} \colon \Vert x_{i} \Vert_{E_{i}} \leq 1 \, , \, i=1, \ldots n\right\}
$. If $E_{1}=\cdots = E_{n}=E$ we will write $\mathcal{L} (^{n} E;F)$ and
whenever $F=\mathbb{C}$ we will simply write $\mathcal{L}(E_{1}, \dots , E_{n})$ or $\mathcal{L} (^{n} E)$.

\smallskip

Ideals of multilinear forms and multilinear operators were introduced by Pietsch in \cite{Pie84}.\ Let us
recall the definition.\ An ideal of multilinear operators $\mathfrak{A}$
is a  subclass of $\mathcal{L}$, the class continuous multilinear operators,
such that, for any  Banach spaces $E_{1}, \dots , E_{n}$ and $F$ the set
\[
\mathfrak{A}(E_{1}, \dots , E_{n};F)=\mathfrak{A} \cap
\mathcal{L}(E_{1}, \dots , E_{n};F)
\]
satisfies
\begin{enumerate}
\item If $S,T \in \mathfrak{A}(E_{1}, \dots , E_{n};F)$, then $S+T \in
\mathfrak{A}(E_{1}, \dots , E_{n};F)$.
\item If $T \in \mathfrak{A}(E_{1}, \dots , E_{n};F)$ and $B_{i} \in
\mathcal{L}(G_{i},E_{i})$ for $i=1, \dots ,n$ and $A\in\B(F;H)$, then $A\circ T \circ
(B_{1},\dots B_{n}) \in \mathfrak{A}(G_{1}, \dots , G_{n};H)$.
\item  The mapping $(x_{1},\dots,x_{n}) \mapsto \gamma_{1}(x_{1})
\cdots \gamma_{n}(x_{n})\cdot f$ belongs to $\mathfrak{A}(E_{1}, \dots ,
E_{n};F)$ for any $\gamma_{1} \in E'_{1}, \dots , \gamma_{n} \in
E'_{n}$ and $f\in F$.
\end{enumerate}
An ideal of multilinear operators is called \emph{normed} if for each
$E_{1},\dots , E_{n}$ and $F$ there is a norm $\| \cdot \|_{\mathfrak{A}(E_{1}, \dots , E_{n};F)}$ in $\mathfrak{A}(E_{1},
\dots , E_{n};F)$ such that
\begin{enumerate}
\item $\|(x_{1},\dots,x_{n}) \mapsto \gamma_{1}(x_{1}) \cdots
\gamma_{n}(x_{n})\cdot f\|_{\mathfrak{A}(E_{1}, \dots , E_{n};F)} =
\|\gamma_{1}\| \cdots \|\gamma_{n}\|\cdot \|f\|$.
\item $\|A\circ T \circ (B_{1},\dots, B_{n}) \|_{\mathfrak{A}(G_{1}, \dots , G_{n};H)} \leq
\|A\| \cdot\|T\|_{\mathfrak{A}(E_{1}, \dots , E_{n};F)} \cdot \|B_{1}\| \cdots
\|B_{n}\|$.
\end{enumerate}
If $\mathfrak{A}(E_{1},\dots , E_{n};F)$ is complete for every Banach spaces $E_{1},\dots , E_{n},F$ we say that
$\mathfrak{A}$ is a \emph{Banach ideal of multilinear operators} (or just, a Banach multilinear ideal).

\emph{The minimal kernel of $\U$} is defined as the composition ideal $\U^{min} := \overline{\mathfrak{F}}\circ\U \circ (\overline{\mathfrak{F}},\dots,\overline{\mathfrak{F}})$, where $\overline{\mathfrak{F}}$ stands for the ideal of approximable operators (i. e. the closure of the ideal of finite rank linear operators).
In other words, a multilinear operator $T$ belongs to $\U^{min}(E_1,\dots,E_n;F)$ if it admits a factorization
\begin{equation} \label{factorminimal}
\xymatrix{ E_1\times\dots\times E_n  \ar[rr]^{\;\;\;\;\;\;\;\;T} \ar[d]_{(B_1,\dots,B_n)} & & F \\
 X_1\times\dots\times X_n\ar[rr]^{\;\;\;\;\;\;\;\;S}  & & Y \ar[u]_A }
\end{equation}
where
$A,B_1,\dots,B_n\in \overline{\mathfrak{F}}$ and $S\in\U(X_1,\dots,X_n;Y)$.
The $\U^{min}$-norm of $T$ is given by $\|T\|_{\U^{min}} := \inf \{\|A\|\cdot  \|S\|_{\U}\cdot \|B_1\|\cdots\|B_n\| \}$, where the infimum runs over all possible factorizations as in~\eqref{factorminimal}.
$\U^{min}$ is the smallest Banach multilinear ideal whose norm coincides with
$\|\cdot\|_{\mathfrak{A}}$ over finite dimensional spaces. This and other properties of $\U^{min}$ can be found in {\cite{Flo01}}. An ideal of multilinear operators is said to be \emph{minimal} if $(\mathfrak{A},\|\cdot\|_{\mathfrak{A}})\iso(\mathfrak{A}^{min},\|\cdot\|_{\mathfrak{A}^{min}})$.

If $\mathfrak{A}$ is a normed ideal of $n$-linear operators, the
\emph{maximal hull} $\mathfrak{A}^{max}$ of $\mathfrak{A}$ is defined as
the class of all $n$-linear operators $T$ such that
\begin{center}
  $\|T\|_{\U^{max}}:=\sup\left\{\left\|Q_L^Y\circ T\circ(I_{M_1}^{X_1},\dots,I_{M_n}^{X_n})\right\|_{\U} : M_i\in FIN(X_i), L\in COFIN(Y)\right\}$
\end{center}
is finite, where $I_M^X:M\rightarrow X$ is the inclusion from $M$ into $X$, $Q_L^Y:Y\rightarrow Y/L$ is the projection of $Y$ over $Y/L$ and $FIN(X)$ ($COFIN(X)$)  represents the class of subspaces of $X$ of finite dimension (codimension).
$\mathfrak{A}^{max}$ is always complete and it
is the largest ideal whose norm coincides with
$\|\cdot\|_{\mathfrak{A}}$ over finite dimensional spaces. A normed
ideal  $\mathfrak{A}$ is called maximal if
$(\mathfrak{A},\|\cdot\|_{\mathfrak{A}})\iso(\mathfrak{A}^{max},\|\cdot\|_{\mathfrak{A}^{max}})$.

If  $\U$ is an ideal of multilinear operators, its \emph{associated tensor norm} is the unique finitely generated tensor norm $\alpha$, of order $n+1$, satisfying $$\U(M_1,\dots, M_n;N) \overset 1 =  (M_1' \otimes \dots \otimes M_n' \otimes N;{\alpha})$$ for every finite dimensional spaces $M_1,\dots,M_n,N$.  In that case we write $\U\sim\alpha$. A detailed study of tensor norms and their relationship with linear/multilinear ideals
 can be found in \cite{DF93, Flo01, Flo02, FG03, FH02}.
Note that $\U$, $\U^{max}$ and $\U^{min}$ have the same associated tensor norm since they coincide isometrically on finite dimensional spaces.

Given a normed ideal $\mathfrak{A}$ associated to a finitely
generated tensor norm $\alpha$, its adjoint ideal $\mathfrak{A}^*$
is defined by
\[\mathfrak{A}^*(E_1,\dots,E_n;F):=\big(E_1\otimes\cdots\otimes E_n\otimes F', \alpha \big)'\cap\B(E_1,\dots,E_n;F).\]
The adjoint ideal is called dual ideal in \cite{Flo01}. The tensor norm associated to $\mathfrak{A}^*$ is denoted by
$\alpha'$. It is well known, by the representation theorem for maximal ideals \cite[Section 4.5]{FH02}, that $\U^*$ is always maximal and $\U^{**}\iso\U^{max}$.

Recall that a  multilinear operator $T\in \B(E_1,\dots,E_n;F)$ is \emph{(Grothendieck) integral} if there exists a regular $F''$-valued Borel
measure $\mu$, of bounded variation on $(B_{E_1'}\times\dots\times B_{E_n'}, w^*)$  such that
$$ T(x_1,\dots,x_n) =\int_{B_{E_1'}\times\dots\times B_{E_n'}} (x_1'(x_1))\cdots (x_n'(x_n)) \ d\mu(x_1',\dots,x_n')$$ for every $x_k\in E_k.$
The space of Grothendieck integral $n$-linear operators is
denoted by $\I(E_1,\dots,E_n;F)$ and the integral norm of a multilinear operator $T$ is defined as $\inf\{ \|\mu\| \}$, where the infimum runs over all the measures $\mu$ representing $T$. This ideal is maximal and its adjoint is the ideal $\B$ of all continuous multilinear mappings.


\section{Interplay between an ideal and its associated sequence space}

 For K\"{o}the sequence spaces $E$ and $F$, an $n$-linear operator $T \in\mathcal{L}(^{n} E; F)$ is said to be \emph{diagonal} if there exists a bounded sequence $\alpha=(\alpha(k))_{k}\in\mathbb C^{\mathbb N}$ such that for all $x_{1}, \dots,x_{n} \in E$ we can write
\begin{equation*}
T(x_{1}, \dots,x_{n}) =\alpha\cdot x_1\cdots x_n= \sum_{k\in\N} \alpha(k)\cdot x_{1}(k) \cdots x_{n}(k)\cdot e_k.
\end{equation*}
In this case, we say that $T$ is the \emph{diagonal multilinear operator associated with} $\alpha$ and we denote it $T_\alpha$.
Given $\U$, an ideal of multilinear operators, we define the \emph{sequence space associated with} $\U$ as
\[\en (\U;E,F):=\big\{\alpha\in \ei : T_\alpha\in\U(^n E; F)\big\}.\]
This is a K\"othe sequence space endowed with the norm $\|\alpha\|_{\en(\U;E,F)}=\|T_\alpha\|_{\U(^nE;F)}$. When $F=\C$ we simply write $\en(\U;E)$.

The following finite-dimensional identifications are easy to check. They
will enable us to prove a duality result next.
\begin{equation}\label{suc finitas}
\boldsymbol{\ell}_n(\U;E,F)_N\overset 1 =\boldsymbol{\ell_n}(\U;E_N,F_N)\overset 1 =\boldsymbol{\ell_n}(\U;E_N^{\times\times},F_N^{\times\times})\overset 1 = \boldsymbol{\ell_n}(\U;E^{\times\times},F^{\times\times})_N.
\end{equation}
\begin{equation}\label{suc finitas max min}
\boldsymbol{\ell_n}(\U;E_N,F_N)\overset 1 =\boldsymbol{\ell_n}(\U^{max};E_N,F_N)\overset 1 =\boldsymbol{\ell_n}(\U^{min};E_N,F_N).
\end{equation}

Our aim is to analyze first the relationship between minimal or maximal ideals with their respective associated sequence spaces and later the interplay between the sequence space associated with an ideal and its adjoint.

 In \cite[Proposition 5.5 and 5.6]{CDS09} it is proved that if $\U$ is a maximal ideal of multilinear forms (scalar valued multilinear operators), then $\en(\U;E)\iso\en(\U;E^{\times\times})$ and $\en(\U;E)\iso\en(\U^*;E^\times)^\times$. In both cases the key of the proofs is the use of \cite[Lemma 5.4]{CDS09}, which is a version of the \emph{Density Lemma} \cite[13.4]{DF93} for diagonal multilinear forms. So, we begin by proving a new version of this Lemma in our vector-valued context and then we can establish some similar results to those given above.

\begin{lemma} \label{density lemma}
Let $\U$ be a maximal ideal of $n$-linear operators and let $E$ and $F$ be K\"{o}the sequence spaces. For a sequence $\alpha$, suppose that there exists a constant $C>0$ such that the projection $\pi_N(\alpha)$ satisfies  $\|\pi_{_N}(\alpha)\|_{\en(\U;E_N,F)}\leq C$ for all $N\in\N$. Then, $\alpha\in\en(\U;E,F^{\times\times})$ and $\|\alpha\|_{\en(\U;E,F^{\times\times})}\leq C$.

In other words, if $T_{\pi_{_N}(\alpha)}\in\U(^nE_N;F)$ has $\U$-norm less than or equal to $C$ for all $N\in\N$ then
$T_\alpha\in\U(^nE;F^{\times\times})$ with $\U$-norm less than or equal to $C$.
\end{lemma}

\begin{proof}
Since $\U$ is maximal, by \cite[Theorem 4.5]{FH02} there exists a finitely generated tensor norm $\nu$ of order $n+1$ such that $\U(^nE;F'')\overset 1= (\otimes^n E\otimes F';\nu)'$. Then, the ball $B_{\U(^nE;F'')}$ is weak-star compact. Thus, the set $\left(J_F\circ T_{\pi_{_N}(\alpha)}\circ (\pi_{_N},\dots,\pi_{_N})\right)_N$, which is contained in the ball $C\cdot B_{\U(^nE;F'')}$, has a weak-star accumulation point  $\Phi\in C\cdot B_{\U(^nE;F'')}$. This mapping should satisfy $\Phi(x_1,\dots,x_n)(e_k')=\alpha(k)\cdot x_1(k)\cdots x_n(k)$, for all $x_1,\dots,x_n\in E$, $e_k'\in F'$.

On the other hand, the canonical mapping $\xi:F''\to F^{\times\times}$ is well defined and has norm less than or equal to 1. Hence, $\xi\circ\Phi$ belongs to $C\cdot B_{\U(^nE;F^{\times\times})}$ and
$$
\xi\circ\Phi(x_1,\dots,x_n)= \left(\Phi(x_1,\dots,x_n)(e_k')\right)_k= \left(\alpha(k)\cdot x_1(k)\cdots x_n(k)\right)_k.
$$
This says that $\xi\circ\Phi$ coincides with the mapping $T_{\alpha}$. In consequence,
$T_\alpha\in\U(^nE;F^{\times\times})$ with $\|T_\alpha\|_{\mathfrak{A}(^{n} E, F^{\times\times})} \leq C$.
\end{proof}

In particular, if $F$ is maximal and $E_0=\spanned\{e_j :  j\in\N\}$, given a diagonal multilinear operator $T_\alpha:E_0\times\dots\times E_0\to F$ such that their truncated operators satisfy $\|T_{\pi_{_N}(\alpha)}\|_{\U(^nE_N;F)}\leq C$ for all $N\in\N$, it follows that $T_\alpha\in\U(^nE;F)$ with $\|T_\alpha\|_{\U(^nE;F)}\leq C$.

 In order to prove the following result, recall the well known characterization of the maximal hull of a sequence space: $x\in E^{max}$ if and only if $\sup_{N\in\N}\|\pi_N(x)\|_E$ is finite, and the norm is given by this supremum. In other words, to ensure that a sequence space $E$ is maximal it is enough to show that if $\|\pi_N(x)\|_E\leq C$ for all $N\in\N$, then $x\in E$ with $\|x\|_E\leq C$.

\begin{proposition}\label{max}
 Let $\U$ be an ideal of $n$-linear operators and let $E$ and $F$ be K\"{o}the sequence spaces.
 \begin{enumerate}
 \item[(i)] If $\U$ is maximal, then $\boldsymbol{\ell_n}(\U;E,F^{\times\times})$ is a maximal K\"othe sequence space and $\boldsymbol{\ell_n}(\U;E,F^{\times\times})\overset 1 = \boldsymbol{\ell_n}(\U; E^{\times\times},F^{\times\times})$.
 \item[(ii)] $\boldsymbol{\ell_n}(\U;E,F)^{max}\overset 1 = \boldsymbol{\ell_n}(\U^{max};E,F^{\times\times}).$
  \end{enumerate}
\end{proposition}

\begin{proof}
(i) Suppose that $\|\pi_{_N}(\alpha)\|_{\boldsymbol{\ell_n}(\U;E,F^{\times\times})}\leq C$ for all $N\in\N$. By  identity \eqref{suc finitas}, we have that $\|\pi_{_N}(\alpha)\|_{\en(\U;E_N,F^{\times\times})}= \|\pi_{_N}(\alpha)\|_{\en(\U;E,F^{\times\times})}\leq C.$
Then, by Lemma~\ref{density lemma}, $\alpha\in\boldsymbol{\ell_n}(\U;E,F^{\times\times})$ with $\|\alpha\|_{\boldsymbol{\ell_n}(\U;E,F^{\times\times})}\leq C$.
So, $\en(\U;E,F^{\times\times})$ is maximal and the same is true for $ \boldsymbol{\ell_n}(\U; E^{\times\times},F^{\times\times})$. Again, identity \eqref{suc finitas}, assures, for each sequence $\alpha$, that $\|\pi_{_N}(\alpha)\|_{\boldsymbol{\ell_n}(\U;E,F^{\times\times})} =\|\pi_{_N}(\alpha)\|_{\boldsymbol{\ell_n}(\U; E^{\times\times},F^{\times\times})},$
for all $N\in\N$. Therefore, $\boldsymbol{\ell_n}(\U;E,F^{\times\times})\overset 1 = \boldsymbol{\ell_n}(\U; E^{\times\times},F^{\times\times})$.

\noindent (ii) For a sequence $\alpha$, by identities \eqref{suc finitas} and \eqref{suc finitas max min}, $\|\pi_{_N}(\alpha)\|_{\boldsymbol{\ell_n}(\U;E,F)}=\|\pi_{_N}(\alpha)\|_{\boldsymbol{\ell_n}(\U^{max};E,F^{\times\times})}$ and  $\boldsymbol{\ell_n}(\U;E,F)^{max}\overset 1 = \boldsymbol{\ell_n}(\U^{max};E,F^{\times\times})^{max}$. Then, applying item (i) we have that $\boldsymbol{\ell_n}(\U^{max};E,F^{\times\times})^{max} \overset 1=\boldsymbol{\ell_n}(\U^{max};E,F^{\times\times})$, which completes the proof.
\end{proof}

\begin{remark} \rm
By the previous proposition, if both the ideal $\U$ and the sequence space $F$ are maximal, then the sequence space $\boldsymbol{\ell_n}(\U;E,F)$ is maximal. Note that the condition over $F$ is necessary for  $\boldsymbol{\ell_n}(\U;E,F)$ to be maximal. Indeed, if $\U=\B$, $E=\ell_\infty$ and $F=c_0$, it follows that $\boldsymbol{\ell_n}(\B;\ell_\infty,c_0)=c_0$, which obviously is not a maximal sequence space.
\end{remark}

\bigskip

Now, we turn to look into the minimal hull.
Recall that a sequence space $E$ is minimal if and only if for all $x\in E$, $\|(\pi_{_N}^E-I_E)(x)\|_E\to0$, or equivalently,  $\pi_N^E$ tends to $I_E$ over compact sets.

\begin{proposition}\label{teo minimal}
Let $\U$ be an ideal of $n$-linear operators and let $E$ and $F$ be K\"{o}the sequence spaces.
 \begin{enumerate}
 \item[(i)] If $\U$ is a minimal ideal and $F$ is a minimal K\"othe sequence space, then $\boldsymbol{\ell_n}(\U;E,F)$ is a minimal sequence space.
 \item[(ii)] $\boldsymbol{\ell_n}(\U;E,F)^{min}\overset 1 = \boldsymbol{\ell_n}(\U^{min};E,F^{min}).$
  \end{enumerate}
\end{proposition}

\begin{proof}
(i) Let $\alpha\in\boldsymbol{\ell_n}(\U;E,F)$. Since $\U$ is minimal, there exist approximable linear operators $A_1,\dots,A_n,B$ and $S\in\U$ such that $T_\alpha=B\circ S\circ(A_1,\dots,A_n)$.
Then,
  \[(\pi_{_N}^F-I_F)\circ T_\alpha=(\pi_{_N}^F-I_F)\circ B\circ S\circ(A_1,\dots,A_n)\in\U(^nE;F),\] and  $\|(\pi_{_N}^F-I_F)\circ T_\alpha\|_{\U(^nE;F)}\leq\|(\pi_{_N}^F-I_F)\circ B\|\cdot\|S\|_{\U}\cdot\|A_1\|\cdots\|A_n\|.$
Now, the mapping $B\in\fin$, which has $F$ as its target space, is compact and $\pi_{_N}^F$ tends to $I_F$ over compact sets (because $F$ is minimal), therefore $\|(\pi_{_N}^F-I_F)\circ B\|$ tends to zero. In consequence, $\|(\pi_{_N}^F-I_F)\circ T_\alpha\|_{\U(^nE;F)}$ tends to zero also
and $\en(\U;E,F)$ is minimal.

\noindent  (ii) For $\alpha\in\boldsymbol{\ell_n}(\U;E,F)^{min}$, the norm $\|\pi_{_N}(\alpha)-\alpha\|_{\boldsymbol{\ell_n}(\U;E,F)}$ tends to zero.
By the identities \eqref{suc finitas} and \eqref{suc finitas max min} we have $\|\pi_{_N}(\alpha)\|_{\boldsymbol{\ell_n}(\U;E,F)}=\|\pi_{_N}(\alpha)\|_{\boldsymbol{\ell_n}(\U^{min};E,F^{min})}.$
Then, $\left(\pi_{_N}(\alpha)\right)_N$ is a Cauchy sequence in the Banach sequence space $\boldsymbol{\ell_n}(\U^{min};E,F^{min})$ and hence it converges to a sequence in $\boldsymbol{\ell_n}(\U^{min};E,F^{min})$ that coincides with $\alpha$ coordinate by coordinate. In other words, this says that $\alpha\in\boldsymbol{\ell_n}(\U^{min};E,F^{min})$ and
\begin{eqnarray*}
\|\alpha\|_{\boldsymbol{\ell_n}(\U^{min};E,F^{min})} &\leq & \|\pi_{_N}(\alpha)-\alpha\|_{\boldsymbol{\ell_n}(\U^{min};E,F^{min})}+\|\pi_{_N}(\alpha)\|_{\boldsymbol{\ell_n}(\U^{min};E,F^{min})}\\
&=& \|\pi_{_N}(\alpha)-\alpha\|_{\boldsymbol{\ell_n}(\U^{min};E,F^{min})}+\|\pi_{_N}(\alpha)\|_{\boldsymbol{\ell_n}(\U;E,F)}\to \|\alpha\|_{\boldsymbol{\ell_n}(\U;E,F)^{min}}.
\end{eqnarray*}
The reverse inclusion holds by item (i).
\end{proof}

Now, we analyze the relationship between the sequence space associated to an ideal with the associated to its adjoint.
Note that when we have a finite dimensional space, its K\"{o}the dual and its classical dual coincide. Moreover, if we call $\nu$  the tensor norm associated to the ideal $\U$, we have,
\begin{center}
$\U^*(^nE_N^\times;F_N^\times)\overset 1 =\left(\bigotimes^nE_N^\times\otimes F_N^{\times\times},\nu\right)'\overset 1 =\U(^nE_N^{\times\times};F_N^{\times\times})'\overset 1 =\U(^nE_N;F_N)'.$
\end{center}
The duality is given in the following way: if $T\in\U(^nE_N;F_N)$ and $S\in\U^*(^nE_N^\times;F_N^\times)$, we can represent them as $T=\sum_{i} x^\times_{i,1}\otimes\cdots\otimes x^\times_{i,n}\otimes y_i$ and $S=\sum_{j} x_{j,1}\otimes\cdots\otimes x_{j,n}\otimes y_j^\times$. Then,
\[\langle S,T \rangle =\sum_{i,j} x^\times_{i,1}(x_{j,1})\cdots x^\times_{i,n}(x_{j,n})\cdot y_j^\times(y_i)\]

It is plain that  $|\langle S,T\rangle|\leq\|S\|_{\U^*(^nE_N^\times;F_N^\times)}\cdot\|T\|_{\U(^nE_N;F_N)}$. Moreover, if $S$ is diagonal, there exists a sequence $\beta$ such that $S=S_\beta=\sum_{j=1}^N \beta(j)\cdot e_j\otimes\cdots\otimes e_j\otimes e_j^\times$. Then,  $\langle S_\beta,T\rangle=\langle S_\beta, D(T)\rangle,$
where  $D(T)=\sum_{i=1}^N T(e_i,\dots,e_i)(i)\cdot e_i^\times\otimes\cdots\otimes e_i^\times\otimes e_i$. A direct argument through this last observation yields to the following result.

\begin{lemma}\label{adjunto finito}
Let $\U$ be an ideal of $n$-linear operators and let $E$ and $F$ be K\"{o}the sequence spaces. Then, $\boldsymbol{\ell_n}(\U;E_N,F_N)^\times\overset 1 =\boldsymbol{\ell_n}(\U^*;E_N^\times,F_N^\times).$
\end{lemma}

As a consequence of the preceding lemma and the identity \eqref{suc finitas} we have
\begin{eqnarray}\label{ln-en}
  \boldsymbol{\ell_n}(\U;E,F)_N^\times  \overset 1  =  \boldsymbol{\ell_n}(\U;E_N,F_N)^\times\overset 1 = \boldsymbol{\ell_n}(\U^*;E_N^\times,F_N^\times)  \overset 1 = \boldsymbol{\ell_n}(\U^*;E^\times,F^\times)_N.
\end{eqnarray}
This equality allows us to give a general result that relates the sequence space associated to an ideal with the corresponding sequence space associated to its adjoint ideal.

\begin{proposition}\label{dual_kothe_adjunto}
Let $\U$ be an ideal of $n$-linear operators and let $E$ and $F$ be K\"{o}the sequence spaces. Then,  $\boldsymbol{\ell_n}(\U;E,F)^\times\overset 1 = \boldsymbol{\ell_n}(\U^*;E^\times,F^\times).$
\end{proposition}

\begin{proof}
For a sequence $\alpha$, by identity \eqref{ln-en}, we have $\left(\boldsymbol{\ell_n}(\U;E,F)^\times\right)^{max}\overset 1 = \boldsymbol{\ell_n}(\U^*;E^\times,F^\times)^{max}$. 
Since a K\"othe dual and an adjoint ideal are maximal, Proposition \ref{max} gives the result.
\end{proof}

Finally, as a consequence of the Propositions  \ref{dual_kothe_adjunto} and \ref{max} and the fact that $\U^{**}\iso\U^{max}$, we obtain the following equalities:
\begin{eqnarray}\label{ln adjunto maximales}
\boldsymbol{\ell_n}(\U^*;E^\times,F^\times)^\times\overset 1 = \boldsymbol{\ell_n}(\U^{**};E^{\times\times},F^{\times\times})
\overset 1 = \boldsymbol{\ell_n}(\U^{\max};E,F^{\times\times}).
\end{eqnarray}

In particular, if $\U$ and $F$ are maximal, $\boldsymbol{\ell_n}(\U;E,F)\overset 1 = \boldsymbol{\ell_n}(\U^*;E^\times,F^\times)^\times$.

\section{Some applications: Multipliers and Lorentz sequence spaces}

An example of a sequence space associated to a set of operators is the  \emph{space of multipliers} from $E$ into $F$, $\M(E,F)$ \cite{DMM02}, which is defined, in our notation, as
\begin{center}
  $\M(E;F)=\boldsymbol{\ell_1}(\B;E,F).$
\end{center}
We begin by showing that our sequence space associated to a multilinear ideal can be seen inside a suitable space of multipliers.

\begin{proposition}\label{prop multiplicadores}
Let $\U$ be an ideal of $n$-linear operators and let $E$ and $F$ be K\"othe sequence spaces. Then, $\boldsymbol{\ell_n}(\U;E,F)\overset{\leq1}\hookrightarrow \M(F^\times,\boldsymbol{\ell_n}(\U;E)).$
\end{proposition}

\begin{proof}
Let $\alpha\in\boldsymbol{\ell_n}(\U;E,F)$ and let $\beta\in F^\times$. Consider $\varphi_\beta\in F'$ given by $\varphi_\beta(x)=\sum_{k\in\N}\beta(k)\cdot x(k)$. If we compose $\varphi_\beta$ with $T_\alpha$, we obtain $\phi_{\alpha\cdot\beta}$ the diagonal $n$-linear form associated to $\alpha\cdot\beta$. Then, by the ideal property, $\phi_{\alpha\cdot\beta}$ belongs to $\U(^nE)$ and
\begin{equation*}
  \|\phi_{\alpha\cdot\beta}\|_{\U(^nE)}\leq\|\varphi_\beta\|_{F'}\cdot\|T_\alpha\|_{\U(^nE;F)}=\|\beta\|_{F^\times}\cdot\|\alpha\|_{\en(\U;E,F)}.
\end{equation*}
In consequence, $\alpha$ belongs to $\M(F^\times,\boldsymbol{\ell_n}(\U;E))$ and $\|\alpha\|_{\M(F^\times,\boldsymbol{\ell_n}(\U;E))}\le \|\alpha\|_{\boldsymbol{\ell_n}(\U;E,F)}$.
\end{proof}

In general, the inclusion given in Proposition \ref{prop multiplicadores} is not an equality. For example, in \cite{CDSV14} it is proved that $\boldsymbol{\ell_n}(\E;\ell_{\frac 3 2},\ell_4)=\ell_4\neq\M(\ell_4^\times,\boldsymbol{\ell_n}(\E;\ell_{\frac3 2}))=\M(\ell_{\frac4 3},\ell_{\frac3 2})=\ell_\infty$, where $\E$ is the ideal of extendible multilinear operators. Another example from the same article is the following:  $\boldsymbol{\ell_n}(\I;\ell_1,\ell_q)=\ell_q\neq\M(\ell_q^\times,\boldsymbol{\ell_n}(\I;\ell_1))=\M(\ell_q^\times,\ei)=\ell_\infty$.

However, for the very particular case of  the ideal of continuous multilinear operators the (isometric) equality holds when  the target set is a K\"{o}the dual.

\begin{proposition}\label{ln-multiplicadores}
   For K\"{o}the sequence spaces $E$ and $F$ we have that
   \begin{center}
     $\boldsymbol{\ell_n}(\B;E,F^{\times})\overset 1 = \M(F,\boldsymbol{\ell_n}(\B;E)).$
   \end{center}
In particular, if $F$ is maximal, $\boldsymbol{\ell_n}(\B;E,F)\overset 1 = \M(F^\times,\boldsymbol{\ell_n}(\B;E)).$
\end{proposition}
\begin{proof}
By Proposition \ref{prop multiplicadores}, $\boldsymbol{\ell_n}(\B;E,F^\times)\overset{\leq1}\hookrightarrow \M(F^{\times\times},\boldsymbol{\ell_n}(\B;E))\overset{\leq1} {\hookrightarrow} \M(F,\boldsymbol{\ell_n}(\B;E)).$
Conversely, let $\alpha\in \M(F,\boldsymbol{\ell_n}(\B,E))$. For any $x_1, \dots, x_n\in E$ and $\beta\in F$, we have
$$
\left|\sum_{k\in\N} \alpha(k)\cdot x_1(k)\cdots x_n(k)\cdot\beta(k)\right|=\left|\phi_{\alpha\cdot\beta}(x_1,\dots x_n)\right|
 \leq \|\phi_{\alpha\cdot\beta}\|_{\B(^nE)}\cdot\|x_1\|_E\cdots\|x_n\|_E.
$$
So, $T_\alpha\in\B(^nE;F^\times)$ and

\noindent$\displaystyle{\|\alpha\|_{\boldsymbol{\ell_n}(\B;E,F^\times)} =  \|T_\alpha\|_{\B(^nE;F^\times)} \le\sup_{\beta\in B_{F}} \|\phi_{\alpha\cdot\beta}\|_{\B(^nE)}
                                =  \sup_{\beta\in B_F}\|\alpha\cdot\beta\|_{\boldsymbol{\ell_n}(\B;E)} = \|\alpha\|_{\M(F,\boldsymbol{\ell_n}(\B;E))}}.
$

Last, if $F$ is maximal, then $\boldsymbol{\ell_n}(\B;E,F)\overset 1 =\boldsymbol{\ell_n}(\B;E,F^{\times\times})\overset 1 = \M(F^\times,\boldsymbol{\ell_n}(\B;E)).$
\end{proof}
 Note that if $F$ is not maximal, the equality $\boldsymbol{\ell_n}(\B;E,F)\overset 1 = \M(F^\times,\boldsymbol{\ell_n}(\B;E))$ might not be true. For instance, take $E=\ell_\infty$ and $F=c_0$, then $\boldsymbol{\ell_n}(\B;\ell_\infty,c_0)=c_0\neq \M(\ell_1;\boldsymbol{\ell_n}(\B;\ell_\infty))=\M(\ell_1,\ell_1)=\ell_\infty.$


\begin{corollary}
   Let $E$ and $F$ be K\"{o}the sequence spaces. Then,
   \begin{center}
     $\boldsymbol{\ell_n}(\I;E,F^{\times})\overset 1 = \boldsymbol{\ell_1}\left(\I;F,\boldsymbol{\ell_n}(\I;E)\right).$
   \end{center}
In particular, if $F$ is maximal, $\boldsymbol{\ell_n}(\I;E,F)\overset 1 = \boldsymbol{\ell_1}\left(\I;F^\times,\boldsymbol{\ell_n}(\I;E)\right).$
\end{corollary}
\begin{proof}
Being $\I^*\iso\B$ and $\B^*\iso \I$, we have
\begin{eqnarray*}
  \en(\I;E,F^\times)&\underset{\eqref{ln adjunto maximales}}{\overset 1 =}&\en(\B;E^\times,F^{\times\times})^\times\underset{\textrm{Prop} \ \ref{ln-multiplicadores}}{\overset 1 =}\boldsymbol{\ell_1}\left(\B;F^\times,\en(\B;E^\times)\right)^\times\\
  &\underset{\textrm{Prop} \ \ref{dual_kothe_adjunto}}{\overset 1 =}&\boldsymbol{\ell_1}\left(\I;F^{\times\times},\en(\B;E^\times)^\times\right)\underset{\textrm{Prop} \ \ref{max}}{\overset 1 =}\boldsymbol{\ell_1}\left(\I;F,\en(\B;E^\times)^\times\right)\\
  &{\overset 1 =}&\boldsymbol{\ell_1}\left(\I;F,\en(\I;E^{\times\times})\right)\ \ {\overset 1 =} \ \ \boldsymbol{\ell_1}\left(\I;F,\en(\I;E)\right),
\end{eqnarray*}
where the last two equalities hold by \cite[Prop 5.5, Prop 5.6]{CDS09}.
\end{proof}

Recall the definition of powers of sequence spaces. Let $E$ be a K\"{o}the sequence space and $0<r<\infty$ such that $\mathbf M^{(\max(1,r))}(E)=1$. Then, $E^r:=\{x\in\ell_\infty \ : \ |x|^{1/r}=(|x(k)|^{1/r})_{k\in\N}\in E \}$
endowed with the norm
  $\|x\|_{E^r}:=\|\ |x|^{1/r}\ \|_E^r$ results a K\"{o}the sequence space which is $\frac{1}{\min(1,r)}$-convex.
And, the sequence space $E^r$ is maximal if $E$ is maximal.

Observe that since $E$ is normal, we can use $x^{1/r}$ instead of $|x|^{1/r}$ in the definition of $E^r$ and its norm.
\begin{remark}\label{norma producto}
Whenever $x_1,\dots,x_n\in B_E$, then  $(x_1\cdots x_n)^{1/n}\in B_E$. Indeed, we have that $\displaystyle{ \big\|(|x_1\cdots x_n|)^{1/n}\big\|_E\leq \Big\|\frac{|x_1|+\cdots +|x_n|}{n}\Big\|_E \leq \frac{\|x_1\|_E+\cdots+\|x_n\|_E}{n}\leq1}$. In particular, for  an $n$-convex K\"othe sequence space $E$, if $x_1,\dots,x_n\in B_E$, then $x_1\cdots x_n\in B_{E^n}$.
\end{remark}

 In the case that $E$ is $n$-convex, there is an alternative description of $\boldsymbol{\ell_n}(\B;E,F)$ as a space of multipliers:

\begin{proposition}\label{lnconvexo}
   Let $E$ and $F$ be an  K\"{o}the sequence spaces such that $E$ is $n$-convex with $\mathbf M^{(n)}(E)=1$. Then, $\boldsymbol{\ell_n}(\B;E,F)\overset 1 = \M(E^n,F).$
\end{proposition}
\begin{proof}
Let $\alpha\in\boldsymbol{\ell_n}(\B;E,F)$ and take $x\in E^n$. Then $x^{1/n}\in E$ and $T_\alpha(x^{1/n},\dots,x^{1/n})=\alpha\cdot x^{1/n}\cdots x^{1/n}=\alpha\cdot x \in F.$
 Thus, $\alpha\in \M(E^n,F)$ and
\begin{align*}
  \|\alpha\|_{\M(E^n,F)}= & \sup_{x\in B_{E^n}} \|\alpha\cdot x\|_F=\sup_{x\in B_{E^n}}\|T_\alpha(x^{1/n},\dots,x^{1/n})\|_F
   \leq  \sup_{x\in B_{E^n}}\|T_\alpha\|_{\B(^nE;F)}\cdot\|x^{1/n}\|^n_E \\
   = & \sup_{x\in B_{E^n}}\|\alpha\|_{\en(\B;E,F)}\cdot\|x\|_{E^n}=\|\alpha\|_{\en(\B;E,F)}.
\end{align*}

Conversely, let $\alpha\in \M(E^n,F)$. Then,   $T_\alpha(x_1,\dots,x_n)=\alpha\cdot x_1\cdots x_n\in F$, for all $x_1,\dots, x_n\in E$.
In consequence, $T_\alpha$ is well defined from $E\times\cdots\times E$ to $F$ and
\begin{center}
  $\displaystyle{\|T_\alpha\|_{\B(^nE;F)} =\sup_{x_i\in B_E}\|\alpha\cdot x_1\cdots x_n\|_F
      \leq  \sup_{x\in B_{E^n}} \|\alpha\cdot x\|_F=\|\alpha\|_{\M(E^n,F)}}.$
\end{center}
\end{proof}

As a direct consequence of Propositions \ref{ln-multiplicadores} and \ref{lnconvexo} we have:

\begin{corollary}
 Let $E$ and $F$ be an  K\"{o}the sequence spaces such that $E$ is $n$-convex with $\mathbf M^{(n)}(E)=1$. Then,  $\M(E^n,F^\times)\overset 1 = \M(F,\boldsymbol{\ell_n}(\B;E)).$
\end{corollary}

Recall the definition of  Lorentz sequence spaces.
For each element $x\in E$ its \emph{decreasing rearrangement} $(x^\star(k))_{k\in\N}$ is given by
$\displaystyle{x^\star(k):=\inf \Big\{ \sup_{j\in\N\backslash J} |x(j)| \  : \ J\subseteq\N,\  card(J)<k \Big\}.}$
Let $(w(k))_{k=1}^{\infty}$ be a decreasing
sequence of positive numbers with $w(1)=1$, $w(k)$ tends to zero
and $\sum_{k=1}^{\infty} w(k)=\infty$ and let $1\leq p<\infty$.\
Then the corresponding Lorentz sequence space, denoted by $d(w,p)$
is defined as the set of all sequences $(x(k))_{k}$ such that
\[
\| x \| = \sup_{\sigma \in \Sigma_{\mathbb{N}}} \bigg(
\sum_{k=1}^{\infty} |x({\sigma(k))}|^{p}\cdot w(k) \bigg)^{1/p} =
\bigg( \sum_{k=1}^{\infty} |x^{\star}(k)|^{p}\cdot w(k) \bigg)^{1/p}
<\infty,
\]
where $\Sigma_{\mathbb{N}}$ denotes the group of permutations of the natural numbers.

The sequence $w$ is said to be $\alpha$-regular ($0< \alpha < \infty$) if $w(k)^{\alpha} \asymp
\frac{1}{k} \sum_{j=1}^{k} w(j)^{\alpha}$ and regular if it is $\alpha$-regular
for some $\alpha$.
In \cite{Rei81} it can be found that the K\"{o}the sequence space $d(w,p)$ is  $r$-convex (with
$\mbox{\bf M}^{(r)} \left(d(w,p)\right) =1$) whenever  $1 \leq r \leq p$.
In \cite{Gar69} and \cite{LTI77} a description of $d(w,p)'$, the dual of $d(w,p)$,
is given.
In the case that $w$ is regular, an easier description of $d(w,p)'$ with $p>1$ is given in \cite{All78,Rei82}.
Let us recall also that, given a strictly positive, increasing sequence
$\Psi$ such that $\Psi(0)=0$, the associated Mar\-cinkiewicz
sequence space $m_\Psi$ (see \cite[Definition 4.1]{KaLee06}, or
\cite{ChHa06,KaLee04}) consists of all sequences $(x(k))_{k}$ such
that
$\displaystyle{\Vert x \Vert_{m_{\Psi}} = \sup_{N} \frac{\sum_{k=1}^{N}
x^{\star}(k)} {\Psi (N)} < \infty .}$

\smallskip

The results of the previous section combined with the scalar-valued case for Lorentz spaces studied  in \cite[Section 5]{CDS09} allow us to give a description of diagonal multilinear mappings from Lorentz sequence spaces (or their duals).

Proposition \ref{ln-multiplicadores} along with  \cite{CDS09} produce
\[\boldsymbol{\ell_n}(\B;d(w,p),F^\times)\overset 1 =\M(F,\boldsymbol{\ell_n}(\B;d(w,p)))\overset 1 =\left\{
                                                                           \begin{array}{ll}
                                                                             \M(F,d(w,p/n)^\times) & \hbox{if $n\leq p$;} \\
                                                                             \M(F,m_\Psi) & \hbox{if $n>p$},
                                                                           \end{array}
                                                                   \right.\]
 where $\Psi(N)=\left(\sum_{k=1}^Nw(k)\right)^{n/p}$. Moreover, $\boldsymbol{\ell_n}(\B;d(w,p),F^\times)=\M(F,\ell_\infty)=\ell_\infty$ if $n>p$ and $w$ is $\frac{n}{n-p}$-regular. For $n\leq p$, since $d(w,p)$ is $n$-convex with $\mathbf M^{(n)}(d(w,p))=1$,  Proposition~\ref{lnconvexo} gives an alternative description: $\boldsymbol{\ell_n}(\B;d(w,p),F)=\M(d(w,p)^n,F)=\M(d(w,p/n),F)$.
Proposition \ref{ln-multiplicadores} combined with some results of \cite{CDS09}, also imply
\[\boldsymbol{\ell_n}(\B;d(w,p)^\times,F^\times)\overset 1 =\M(F,\boldsymbol{\ell_n}(\B;d(w,p)^\times))\overset 1 =\left\{
                                                                           \begin{array}{ll}
                                                                             \M(F,\ell_\infty)=\ell_\infty & \hbox{if $n'\leq p$;} \\
                                                                             \M(F,d(w^{\frac{n'}{n'-p}},\frac{p'}{p'-n})) & \hbox{if $1<p<n'$}\\
                                                                             \M(F,d(w^n,1)) & \hbox{if $p=1$}
                                                                           \end{array}
                                                                   \right.\]
To complete this description it remains to calculate the space of multipliers from $F$ to a Lorentz sequence space. We can give an explicit characterization when $F=\eq$.
We affirm that
\[\M\left(\ell_q,d(w,p)\right)=\left\{
                     \begin{array}{ll}
                       d\left(w^{\frac{q}{q-p}},\frac{pq}{q-p}\right) & \hbox{si $p<q$;} \\
                       \ \ \ \  \ell_\infty & \hbox{si $p\geq q$.}
                     \end{array}
                   \right.\]

\noindent Indeed, when $p\geq q$, the equality is clear from the inclusions $\ell_q\subseteq\ell_p\subseteq d(w,p)$.

When $p<q$,
\begin{eqnarray*}
 \|\alpha\|_{\M\left(\ell_q,d(w,p)\right)} & = & \sup_{\beta\in B_{\ell_q}}\|\alpha \cdot\beta\|_{d(w,p)}= \sup_{\beta\in B_{\ell_q}}\sup_{\sigma\in\Sigma_\N} \left( \sum_{k\in\N}|\alpha_\sigma(k)|^p\cdot|\beta_\sigma(k)|^p\cdot w(k)\right)^{1/p}\\
& = & \sup_{\sigma\in\Sigma_\N} \sup_{\gamma\in B_{\ell_{\frac q p}}} \left( \sum_{k\in\N}|\alpha_\sigma(k)|^p\cdot|\gamma(k)|\cdot w(k)\right)^{1/p}= \sup_{\sigma\in\Sigma_\N} \left( \|\alpha_\sigma^p\cdot w\|_{(\frac q p)'}\right)^{1/p}\\
 & = & \sup_{\sigma\in\Sigma_\N} \left( \sum_{k\in\N}|\alpha_\sigma(k)|^{p(\frac q p)'}\cdot w(k)^{(\frac q p)'}\right)^{\frac 1 {p(\frac q p)'}}= \|\alpha\|_{d\left(w^{\frac{q}{q-p}},\frac{pq}{q-p}\right)}.
\end{eqnarray*}

 We can obtain, applying Theorem \ref{dual_kothe_adjunto} and taking into account that $\I^*=\B$, similar results for the ideal of integral multilinear operators.

\section{(E,p)-summing multilinear operators}

The classical notion of $(q,p)$-summing operator has a natural extension by changing the index $q$ (which refers to the space $\ell_q$) by any other K\"othe sequence space $E$ containing $\ell_p$. This yields the concept of $(E,p)$-summing linear mapping. This class, denoted by $\Pi_{(E,p)}$, was studied in \cite{DMM02} where typical results about $(q,p)$-summing operators are extended to the case of $(E,p)$-summing linear mappings by means of the space of multipliers. Here, we propose an $n$-linear version of that program.

Along this section we consider $1\leq p<\infty$. If $\ell_p\hookrightarrow E$, we denote by  $c_p^E=\|i:\ell_p\hookrightarrow E\|$  the norm of the natural inclusion map. We need to recall also that for a K\"othe sequence space $E$, the $1/n$-convexification, $E^{1/n}$, is always well defined and it is an $n$-convex K\"othe sequence space with $\mathbf M^{(n)}(E^{1/n})=1$. Now, we can proceed to the definition.

 Let $E$ be a K\"{o}the sequence space such that $\ep\hookrightarrow E^{1/n}$ and let $X_1,\dots,X_n,Y$ be Banach spaces. An $n$-linear operator $T\in\B(X_1,\dots,X_n;Y)$ is called \emph{$(E,p)$-summing}  if there exists $C>0$ such that for every finite sequences $x_1=(x_{1,i})_{i=1}^m\subseteq X_1$,\dots, $x_n=(x_{n,i})_{i=1}^m\subseteq X_n$ it holds
\begin{eqnarray}\label{(E,p)-dominado}
  \bigg\|\big(\left\|T(x_{1,i},\dots,x_{n,i})\right\|_Y\big)_{i=1}^m\bigg\|_{E}\leq C\cdot c_p^{E^{1/n}}\cdot w_p(x_1)\cdots w_p(x_n),
\end{eqnarray}
where $w_p(x)=\sup_{x'\in B_{E'}}\big(\sum_{i=1}^m |\langle x',x_i\rangle|^p\big)^{1/p}$ is the weak $\ep$-norm.
 The space of $(E,p)$-summing $n$-linear operators from $X_1\times\cdots\times X_n$ to $Y$ is denoted by $\Pi_{(E,p)}(X_1,\dots,X_n;Y)$. It is a Banach space endowed with the norm $\pi_{(E,p)}(T)=\inf\{C>0 \ : \ T \textrm{\ verifies \ } \eqref{(E,p)-dominado}\}$. Moreover, it is easy to see that $\Pi_{(E,p)}$ is a Banach ideal of $n$-linear operators (always  under the  condition $\ep\hookrightarrow E^{1/n}$).

When $E=\ell_q$ and $p\le nq$, the ideal $\Pi_{(E,p)}$ is the class of $(q,p)$-summing $n$-linear mappings   $\Pi_{(q,p)}$ introduced and studied by \cite{Mat93}. For $p=nq$, $(q,p)$-summing $n$-linear mappings are the so called $p$-dominated $n$-linear mappings $\mathcal D_p$ \cite{Mat93,sch91,MT99}.
 In the case $n=1$, the class $\Pi_{(E,p)}$ is the usual ideal of $(E,p)$-summing linear operators mentioned above. When $E=\ep$ this is just the classical ideal of absolutely $p$-summing linear mappings, $\Pi_p$, with $\pi_p(\cdot)$ as the usual notation for its norm.

 In the sequel we present $n$-linear versions of some results in \cite{DMM02} along with a relationship between the sequence space associated to the ideal $\Pi_{(E,p)}$ and a linear relative.

 We begin by the $n$-linear version of  \cite[Lemma 3.3]{DMM02}. It is a standard characterization of $(E,p)$-summability with a straightforward proof that we omit.

 \begin{lemma} \label{eq (E,p)}
Let $E$ be a K\"{o}the sequence space such that $\ep\hookrightarrow E^{1/n}$ and let $X_1,\dots,X_n,Y$ be Banach spaces. For a mapping $T\in\B(X_1,\dots,X_n;Y)$  and a constant $C\ge 0$ the following are equivalent:
  \begin{itemize}
    \item[(1)] $T\in\Pi_{(E,p)}(X_1,\dots,X_n;Y)$  with $\pi_{(E,p)}(T)\leq C$.
    \item[(2)] $\pi_{(E,p)}\left(T\circ(A_1,\dots,A_n)\right)\leq C$ for all $m\in\N$ and for all $A_j\in \B(\ell_{p'}^m;X_j)$ with $\|A_j\|\leq 1$. (Here $\ell_{p'}^m$ means the space $\C^m$ with the $\ell_{p'}$-norm, not to be confused with the $m$-power of the sequence space $\ell_{p'}$.)
  \end{itemize}

In particular, in this case,
$$
\pi_{(E,p)}(T)=\sup_m\left\{\pi_{(E,p)}\left(T\circ(A_1,\dots,A_n)\right): \|A_j\|_{\B(\ell_{p'}^m;X_j)}\leq 1, \textrm{ for } j=1,\dots,n\right\}.
$$
\end{lemma}

Next lemma enumerates two simple properties about a sequence space associated to an ideal that will be needed later.

\begin{lemma}
  Let $E$ be a K\"{o}the sequence space  with $\ep\hookrightarrow E^{1/n}$.
  \begin{enumerate}
 \item[(i)] If $q$  and $r$  are such that $p<q$ and $\frac 1 r=\frac 1 p-\frac 1 q$, then $\eq\hookrightarrow\left[\en(\B;\ell_r,E)\right]^{1/n}$.
\item[(ii)]  If $F$ and $G$ are K\"{o}the sequence spaces  and $G$ is $n$-convex with $\mathbf M^{(n)}(G)=1$, then $\boldsymbol{\ell_1}(\Pi_{(E^{1/n},p)};F,G)$ is $n$-convex with convexity constant  1.
  \end{enumerate}
\end{lemma}
\begin{proof}
(i) First, note that $\alpha\in\left[\en(\B;\ell_r,E)\right]^{1/n}$ if and only if $\left(\alpha\cdot x_1^{1/n}\cdots x_n^{1/n}\right)\in E^{1/n}$ for all $x_1,\dots,x_n\in\ell_r$. Now, for $\alpha\in\eq$ and $x_1,\dots, x_n\in\ell_r$, it is clear that $x_1^{1/n}\cdots x_n^{1/n}\in\ell_r$ and so $\alpha\cdot x_1^{1/n}\cdots x_n^{1/n}\in\ep\hookrightarrow E^{1/n}$.

\noindent(ii)
Let $\alpha_1,\dots,\alpha_N\in\boldsymbol{\ell_1}(\Pi_{(E^{1/n},p)};F,G)$, we have to show that
\[\left\|\bigg[\Big(\sum_{k=1}^N|\alpha_k(j)|^n\Big)^{1/n}\bigg]_{j=1}^\infty\right\|_{\ell_1(\Pi_{(E^{1/n},p)};F,G)}\leq \left(\sum_{k=1}^N\|\alpha_k\|_{\ell_1(\Pi_{(E^{1/n},p)};F,G)}^n\right)^{1/n}.\]
Equivalently, if we call $\beta(j)=\left(\sum_{k=1}^N|\alpha_k(j)|^n\right)^{1/n}$, the condition to be checked is
\[\|D_\beta\|_{\Pi_{(E^{1/n},p)}(F;G)}\leq\left(\sum_{k=1}^N\|D_{\alpha_k}\|_{\Pi_{(E^{1/n},p)}(F;G)}^n\right)^{1/n}.\]
Now, let $x_1,\dots,x_m\in F$. Since $G$ is $n$-convex with $\mathbf M^{(n)}(G)=1$  we obtain
\begin{align*}
  \bigg\|\Big(\|D_\beta(x_i)\|_G\Big)_{i=1}^m & \bigg\|_{E^{1/n}} =  \bigg\|\Big(\|\beta\cdot x_i\|_G\Big)_{i=1}^m\bigg\|_{E^{1/n}} = \bigg\|\Big(\|\beta\cdot x_i\|_G^n\Big)_{i=1}^m\bigg\|_E^{1/n}\\
  = & \left\|\left(\bigg\|\Big(\sum_{k=1}^N |\alpha_k\cdot x_i|^n\Big)^{1/n}\bigg\|_G^n\right)_{i=1}^m\right\|_E^{1/n}\leq \left\|\bigg(\sum_{k=1}^N \|\alpha_k\cdot x_i\|_G^n\bigg)_{i=1}^m\right\|_E^{1/n}\\
 (\textrm{by Minkowski}) \ \ \leq & \left(\sum_{k=1}^N\bigg\|\Big(\|D_{\alpha_k}(x_i)\|_G^n\Big)_{i=1}^m\bigg\|_E\right)^{1/n}
 =  \left(\sum_{k=1}^N\bigg\|\Big(\|D_{\alpha_k}(x_i)\|_G\Big)_{i=1}^m\bigg\|_{E^{1/n}}^n\right)^{1/n}\\
 \leq &  \left(\sum_{k=1}^N \Big( \|D_{\alpha_k}\|_{\Pi_{(E^{1/n},p)}}\cdot w_p(x_i) \Big)^n\right)^{\frac 1 n}
 =  \left(\sum_{k=1}^N \|D_{\alpha_k}\|_{\Pi_{(E^{1/n},p)}}^n\right)^{\frac 1 n}\cdot w_p(x).
 \end{align*}
Then,
$\|D_\beta\|_{\Pi_{(E^{1/n},p)}(F;G)}\leq\left(\sum_{k=1}^N\|D_{\alpha_k}\|_{\Pi_{(E^{1/n},p)}(F;G)}^n\right)^{1/n}$.
\end{proof}

Note that under the assumptions of  item (i) of lemma above  the class $\Pi_{\left(\en(\B;\ell_r,E), q\right)}$ for  $n$-linear operators is well defined.

\begin{remark}\label{norma producto2}
In Remark \ref{norma producto} we showed that whenever $\|x_1\|_E=\cdots=\|x_n\|_E=1$, then $\|(x_1\cdots x_n)^{1/n}\|_E^n\leq1$. Thus, for every $x_1,\dots, x_n\in E$ we have $\|(x_1\cdots x_n)^{1/n}\|_E^n\leq\|x_1\|_E\cdots\|x_n\|_E$. Hence, if the space $E$ is $n$-convex, $\|x_1\cdots x_n\|_{E^{n}}\leq\|x_1\|_E\cdots\|x_n\|_E$. In particular, it holds that
$\|x_1\cdots x_n\|_E\leq\|x_1\|_{E^{1/n}}\cdots\|x_n\|_{E^{1/n}}$,
for all $x_1,\dots,x_n\in E^{1/n}$.
\end{remark}

Now we present in the next theorem two composition results about $(E,p)$-summing $n$-linear mappings.  Since the statement involve both linear and $n$-linear ideals to avoid confusion we chose to denote by $\Pi_{(E,p)}^{(n)}$ the  ideal of $(E,p)$-summing $n$-linear operators. Observe that for the particular case of $E=\ell_{\frac p n}$ both compositions are known results about $p$-dominated $n$-linear operators \cite{MT99}.

\begin{theorem}[Composition theorem for $(E,p)$-summing multilinear mappings]\label{teo de composicion (E,p)}

\noindent Let $E$ be a K\"{o}the sequence space  such that  $\ep\hookrightarrow E^{1/n}$.
  \begin{enumerate}
 \item[(i)] If $q$  and $r$  are such that $p<q$ and $\frac 1 r=\frac 1 p-\frac 1 q$, then
\[\Pi_{\left(\en(\B;\ell_r,E), q\right)}^{(n)}\circ(\Pi_r ,\dots,\Pi_r)\subseteq\Pi_{(E, p)}^{(n)}.\]
Moreover, $\pi_{(E,p)}^{(n)}\left(T\circ(A_1,\dots,A_n)\right)\leq \pi_{(\en(\B;\ell_r,E),q)}^{(n)}(T)\cdot\pi_r(A_1)\cdots\pi_r(A_n)$, for $T$ belonging to $\Pi_{(\en(\B;\ell_r,E),q)}^{(n)}$ and $A_1, \dots, A_n$ belonging to $\Pi_r$.
\item[(ii)] It holds
$$\mathcal L^{(n)}\circ(\Pi_{(E^{1/n},p)} ,\dots,\Pi_{(E^{1/n},p)})\subseteq\Pi_{(E, p)}^{(n)}.$$
Moreover,
$\pi_{(E,p)}^{(n)}(T\circ(A_1,\dots,A_n))\leq\|T\|\cdot \pi_{(E^{1/n},p)}(A_1)\cdots\pi_{(E^{1/n},p)}(A_n)$ for $T$ an $n$-linear operator and $A_1, \dots, A_n$ belonging to $\Pi_{(E^{1/n},p)}$.
\end{enumerate}
\end{theorem}

\begin{proof}
(i) This is an $n$-linear version of \cite[Lemma 3.5]{DMM02}. The proof is similar so we omit it.

\noindent (ii) Let $x_1=(x_{1,i})_{i=1}^m\subseteq X_1$,\dots, $x_n=(x_{n,i})_{i=1}^m\subseteq X_n$. By the normal property of $E$  and by  Remark \ref{norma producto2}, we have
  \begin{align*}
    \bigg\|\Big(\big\|T(A_1(x_{1,i}),\dots,A_n(x_{n,i}))&\big\|_Y\Big)_{i=1}^m\bigg\|_E  \leq \bigg\|\Big(\|T\|\cdot\|A_1(x_{1,i})\|_{Y_1}\cdots\|A_n(x_{n,i})\|_{Y_n}\Big)_{i=1}^m\bigg\|_E \\
     & = \|T\|\cdot\bigg\|\Big(\|A_1(x_{1,i})\|_{Y_1}\cdots\|A_n(x_{n,i})\|_{Y_n}\Big)_{i=1}^m\bigg\|_E \\
     & \leq \|T\|\cdot\bigg\|\Big(\|A_1(x_{1,i})\|_{Y_1}\Big)_{i1}^m\bigg\|_{E^{1/n}}\cdots\bigg\|\Big(\|A_n(x_{n,i})\|_{Y_n}\Big)_{i=1}^m\bigg\|_{E^{1/n}}\\
     &\leq \|T\|\cdot \pi_{(E^{1/n},p)}(A_1)\cdot w_p(x_1)\cdots\pi_{(E^{1/n},p)}(A_n)\cdot w_p(x_n)
  \end{align*}
\end{proof}

The next proposition shows that sequence space associated to $(E,p)$-summing $n$-linear operators can be seen as the $n$-convexification of a sequence space associated to $(E^{1/n},p)$-summing linear mappings. This identification extend an analogous result for scalar-valued $p$-dominated $n$-linear mappings proved in \cite{CDS06} (see explanation below).

\begin{proposition}
  Let $E$, $F$ and $G$ be K\"{o}the sequence spaces such that $\ep\hookrightarrow E^{1/n}$. Then $$\en\big(\Pi_{(E,p)};F,G\big)\overset 1 = \big[\boldsymbol\ell_1\left(\Pi_{(E^{1/n},p)};F,G^{1/n}\right)\big]^n.$$
\end{proposition}

\begin{proof}
  Note first that $\alpha\in\big[\boldsymbol\ell_1\left(\Pi_{(E^{1/n},p)};F,G^{1/n}\right)\big]^n$ if and only if the diagonal linear operator $D_{\alpha^{1/n}}\in\Pi_{(E^{1/n},p)}(F;G^{1/n})$. Let $\alpha\in\en\big(\Pi_{(E,p)};F,G\big)$ and take $x_1,\dots,x_m\in F$, then
  \begin{align*}
    \Bigg\|\bigg(\big\|&D_{\alpha^{1/n}}(x_i)\big\|_{G^{1/n}}  \bigg)_{i=1}^m\Bigg\|_{E^{1/n}} = \left\|\bigg(\big\|(\alpha^{1/n}\cdot x_i)^n\big\|_G^{1/n}\bigg)_{i=1}^m\right\|_{E^{1/n}}\\
      & = \left\|\bigg(\big\|T_\alpha(x_i,\dots,x_i)\big\|_G^{1/n}\bigg)_{i=1}^m\right\|_{E^{1/n}}
    =  \left\|\bigg(\big\|T_\alpha(x_i,\dots,x_i)\big\|_G\bigg)_{i=1}^m\right\|_E^{1/n} \\
    & \leq \left(\|T_\alpha\|_{\Pi_{(E,p)}}\cdot w_p(x)^n\right)^{1/n}=\|\alpha\|_{\en\big(\Pi_{(E,p)};F,G\big)}^{1/n}\cdot w_p(x).
  \end{align*}
  Then, $\alpha\in\big[\boldsymbol\ell_1\left(\Pi_{(E^{1/n},p)};F,G^{1/n}\right)\big]^n$ and $$\|\alpha\|_{\big[\boldsymbol\ell_1\left(\Pi_{(E^{1/n},p)};F,G^{1/n}\right)\big]^n}
=\|D_{\alpha^{1/n}}\|_{\Pi_{(E^{1/n},p)}(F;G^{1/n})}^n  \leq \|\alpha\|_{\en\big(\Pi_{(E,p)};F,G\big)}.$$

Conversely, let $\alpha\in\big[\boldsymbol\ell_1\left(\Pi_{(E^{1/n},p)};F,G^{1/n}\right)\big]^n$. Consider the factorization of $T_\alpha=\Psi\circ(D_{\alpha^{1/n}},\dots,D_{\alpha^{1/n}})$, where the operator  $\Psi\in\B(^nG^{1/n};G)$ is given by $\Psi(x_1,\dots,x_n)=x_1\cdots x_n$, and $D_{\alpha^{1/n}}\in\Pi_{(E^{1/n},p)}(F;G^{1/n})$. Applying Theorem \ref{teo de composicion (E,p)} (2), we obtain that $T_\alpha\in\Pi_{(E,p)}(^nF;G)$ and $\pi_{(E,p)}(T_\alpha)\leq\|\Psi\|\cdot \big(\pi_{(E^{1/n},p)}(D_{\alpha^{1/n}})\big)^n\leq \|\alpha\|_{\big[\boldsymbol\ell_1\left(\Pi_{(E^{1/n},p)};F,G^{1/n}\right)\big]^n}$.
\end{proof}

Some comments are in order. As we have mentioned, if $p\ge n$ and $E=\ell_{\frac p n}$ then for $n$-linear mappings $\Pi_{(E,p)}$ coincides with $\mathcal D_p$ (the ideal of $p$-dominated mappings). For this particular case, the identity of the previous proposition reads as follows:
$$
\en\big(\mathcal D_p;F,G\big)\overset 1 = \big[\boldsymbol\ell_1\left(\Pi_p;F,G^{1/n}\right)\big]^n.
$$
This can be seen as the vector-valued version of \cite[Prop. 2.1]{CDS06} which, translated to our current terminology says (for $n\ge 2$):
$$
\en\big(\mathcal D_p;F\big)\overset 1 = \big[\boldsymbol\ell_1\left(\Pi_p;F,\ell_n\right)\big]^n.
$$
Actually, that result was just for $F$ an $\ell_p$ space, but the same argument works for any K\"{o}the sequence space. Moreover, it can also be proved, following analogous arguments that (when $\ep\hookrightarrow E^{1/n}$ and $n\ge 2$)
$$\en\big(\Pi_{(E,p)};F\big)\overset 1 = \big[\boldsymbol\ell_1\left(\Pi_{(E^{1/n},p)};F,\ell_n\right)\big]^n.$$

As an interesting consequence of these identities we derive, for every $n\ge 2$,
$$\en\big(\Pi_{(E,p)};F,\ell_1\big)\overset 1 = \en\big(\Pi_{(E,p)};F\big).$$
Note that clearly this equality holds also for the ideal $\mathcal L$: for  $n\ge 2$,
$\en\big(\mathcal L;F,\ell_1\big)\overset 1 = \en\big(\mathcal L;F\big).$
However it is not true for any ideal $\mathfrak A$. For instance, for the ideal $\mathcal I$ of integral multilinear mappings we know, from identity \eqref{ln adjunto maximales} that
$\en\big(\mathcal I; E, \ell_1\big)\overset 1 = \en \big(\mathcal L;E^{\times},\ell_\infty\big)^{\times} \overset 1 =\ell_1$, for any K\"{o}the sequence space $E$. But $\en\big(\mathcal I; E\big)$ is not always equal to $\ell_1$. Indeed, by \cite[Prop. 1.2]{CDS06} (see also \cite{CDSV14}), $\en\big(\mathcal I; \ell_p\big)\overset 1 =\ell_{\frac{p'}n}$, for $1\le p\le\frac{n}{n-1}$.

Finally, we extend to the multilinear setting an Inclusion theorem for $(E,p)$-summing operators proved in \cite[Lemma 3.4]{DMM02}. The proof is similar, so we omit it. For $E=\ell_{\frac p n}$, the first inclusion is just the usual inclusion of $p$-dominated  into $q$-dominated $n$-linear operators when $p<q$. Other inclusion results about $(q,p)$-summing multilinear mappings with and without hypothesis about cotype 2 spaces in the domain can be found in \cite{BMP10}.

\begin{theorem}[Inclusion theorem for $(E,p)$-summing multilinear operators]\label{teo de inclusion (E,p)}
\

\noindent Let $E$ be a K\"{o}the sequence space such that  $\ep\hookrightarrow E^{1/n}$. If $q$  and $r$  are such that $p<q$ and $\frac 1 r=\frac 1 p-\frac 1 q$, then we have the following inclusion for ideals of $n$-linear mappings:
\[\Pi_{(E, p)}\subseteq\Pi_{\left(\en(\B;\ell_r,E), q\right)},\]
with $\pi_{(\en(\B;\ell_r,E),q)}(T)\leq c_p^{E^{1/n}}\cdot \left( c_q^{\en(\B;\ell_r,E)}\right)^{-1}\cdot \pi_{(E,p)}(T)$, for all $T\in\Pi_{(E,p)}^{(n)}$.

Moreover, if $X_1,\dots,X_n$ have cotype 2, then for any Banach space $Y$,
\begin{eqnarray}\label{identidad 1,2-dominado}
  \Pi_{\left(\en(\B;\ell_{2},E),2\right)}(X_1,\dots,X_n;Y)=\Pi_{(E,1)}(X_1,\dots,X_n;Y).
\end{eqnarray}
\end{theorem}

In the same spirit of the definition of $(E,p)$-summing multilinear operators and having in mind the concept of strongly $p$-summing multilinear operators \cite{Dim03}, we introduce the class of strongly $(E,p)$-summing multilinear operators.
Let $E$ be a K\"{o}the sequence space such that $\ep\hookrightarrow E$. An $n$-linear operator $T\in\B(X_1,\dots,X_n;Y)$ is said to be \emph{strongly $(E,p)$-summing} if exists $C>0$ such that for finite sequences $(x_{1,i})_{i=1}^m\subseteq X_1$, \dots,$(x_{n,i})_{i=1}^m\subseteq X_n$,  it satisfies that
\begin{eqnarray}\label{fuertemente (E,p)-sumante}
 \ \ \ \ \ \ \left\|\left(\left\|T(x_{1,i},\dots,x_{n,i})\right\|_Y\right)_{i=1}^m\right\|_{E}\leq C\cdot c_p^E\cdot \sup_{\phi\in B_{\B(X_1,\dots,X_n)}} \left(\sum_{i=1}^m |\phi(x_{1,i},\dots,x_{n,i})|^p\right)^{1/p}.
\end{eqnarray}
\noindent We note $\s_{(E,p)}$ the space of strongly $(E,p)$-summing multilinear operators. It is easy to see that it is an ideal of $n$-linear operators endowed with the norm
\[S_{(E,p)}(T)=\inf\{C \ : \ T \textrm{\ verifies \ } \eqref{fuertemente (E,p)-sumante}\}.\]
Applying the same arguments used in \cite{DMM02} to prove the inclusion theorem for $(E,p)$-summing linear operators, it can be proved an analogous inclusion theorem for strongly $(E,p)$-summing multilinear operators.
\begin{theorem}
  Let $E$ be a K\"{o}the sequence space such that $\ep\hookrightarrow E$. If $\frac n r=\frac 1 p-\frac 1 q$, then
\[\s_{(E, p)}\subseteq\s_{\left(\en(\B;\ell_r,E), q\right)}.\]
Moreover, if $T\in\s_{(E,p)}$, then $S_{(\en(\B;\ell_r,E),q)}(T)\leq c_p^E\cdot \left( c_q^{\en(\B;\ell_r,E)}\right)^{-1}\cdot S_{(E,p)}(T)$.
\end{theorem}

\subsection*{Acknowledgements} We would like to thank Daniel Carando for helpful conversations and for suggesting some of the problems developed in the article. We also want to thank the anonymous referee for useful comments that led to a better presentation of Section 4.

\end{document}